	\documentclass[12pt,a4paper,oneside]{article}

	\usepackage{amsmath, amsthm, amsfonts, amssymb}
	\usepackage{enumerate}
	\usepackage{graphicx}
	\usepackage[all,2cell,v2]{xy}\UseTwocells\UseHalfTwocells

	\swapnumbers

	\theoremstyle{plain}
		\newtheorem{theorem}		{Theorem}[section]
		\newtheorem{proposition}	[theorem]{Proposition}
		\newtheorem{corollary}		[theorem]{Corollary}
		\newtheorem{lemma}			[theorem]{Lemma}
	\theoremstyle{definition}
		\newtheorem{definition}		[theorem]{Definition}
	\theoremstyle{remark}
		\newtheorem{remark}		[theorem]{Remark}
		\newtheorem{example}		[theorem]{Example}

		\def\DC{\Delta/\!/C}
		\def\laxto{\dasharrow}
		\def\xto{\xrightarrow}
		\def\from{\leftarrow}
		\def\then{\Rightarrow}
		\def\op{^{\circ}}
		\def\u{\underline}
		\def\ord#1{[#1]}
		\def\tensor{\otimes}

		\def\cat	{{\rm Cat}}
		
		\def\sset	{{\rm SSet}}
		\def\catl	{\widetilde{\cat}}
		\def\top	{{\rm Top}}

		\def\ob{{\rm ob}}
		\def\fl{{\rm ar}}
		\def\id{{\rm id}}
		\def\ho{{\rm ho}}
		\def\pr{{\rm pr}}
		\def\diag{{\rm diag}}
		\def\cte{c}
		\def\sd{{\rm sd}}
		\def\eps{\epsilon}
		\def\geq{\geqslant}
		\def\leq{\leqslant}

    \def\title#1{\noindent{\bf\LARGE{#1}} \bigskip \thispagestyle{plain}}
    \def\author#1{\noindent{\sc #1}\smallskip}
    \def\address#1{\noindent #1}

\begin{document}

\title{On the loop space of a 2-category}

\author{Matias L. del Hoyo}

\address{Departamento de Matem\'atica\\FCEyN, Universidad de Buenos Aires\\Buenos Aires, Argentina.}

\thispagestyle{empty}

\begin{abstract}
Every small category $C$ has a {\em classifying space} $BC$ associated in a natural way. This construction can be extended to other contexts and set up a fruitful interaction between categorical structures and homotopy types. In this paper we study the classifying space $B_2C$ of a 2-category $C$ and prove that, under certain conditions, the loop space $\Omega_c B_2C$ can be recovered {\em up to homotopy} from the endomorphisms of a given object.
We also present several subsidiary results that we develop to prove our main theorem. 
\end{abstract}

{\footnotesize
}

\section*{Introduction}

The construction of classifying spaces for small categories was introduced by Segal \cite{segal}, following ideas of Grothendieck and generalizing Milnor's construction for principal $G$-bundles. This theory was developed by Segal, Quillen and Thomason among others, with remarkable applications in K-theory and abstract homotopy theory \cite{quillen,segal2,thomason}. Lately, the construction of classifying spaces has been extended to other categorical structures, such as 2-categories and fibred categories \cite{bc,dh2}.

\smallskip

Given $C$ a small 2-category, let us denote by $B_2C$ its classifying space, defined in section \ref{B_2C}. The present work is motivated by Theorem \ref{main}, which asserts that under certain conditions there is a homotopy equivalence
$$\Omega_cB_2C\simeq B(C(c,c)).$$
Here $c$ denotes a fixed object of $C$, $\Omega_cB_2C$ is the space of loops of $B_2C$ with basepoint $c$, and $C(c,c)$ is the category of endomorphisms of $c$ in $C$, which play the role of {\it algebraic loops}.
This theorem can be thought of as a formulation of the classical ideas in delooping monoidal categories \cite{segal2,thomason,minian}.

\smallskip

Throughout this paper we carry out a number of technical developments that may have interest in themselves.
In section \ref{subdiv} we give a new formulation for the categorical subdivision $\sd:\cat\to\cat$, which simplifies the definitions and proofs when dealing with this functor.
Then we use the subdivision to solve the following categorical problem: given $C$ in $\cat$, construct $\tilde C$ in $2\cat$ and a lax functor $C\laxto\tilde C$ universal for this property (theorem \ref{up}).
Later on, we apply this theorem to prove a version of Quillen's Theorem A for lax functors (theorem \ref{laxthma}).
Finally, we associate to every 2-category $C$ its category of simplices $\DC$ and a lax functor $\DC\laxto C$ in a natural way, %which is some kind of {\em resolution}. We 
and show that this map is a weak equivalence (theorem \ref{2suplaxo}).

\smallskip

Using these results we set up a categorical analogue of the path fibration of spaces
$$C(c,c)\op \to E  \to \DC$$
where $E$ stands for the (opposite category to the) Grothendieck construction over the path functor $L:(\DC)\op\laxto\cat$ (definition \ref{path}). 
We prove theorem \ref{main} by applying Quillen's Theorem B to this fibration.

\bigskip

\noindent{\bf Organization}

The first three sections contain preliminaries. In section 1 we recall the basics on classifying spaces of categories, plus a quick review of fibred categories. Section 2 is a summary of  2-categories and lax functors. We overview the classifying space of a 2-category in section 3.

Later on, we concentrate on the involved technical aspects. Section 4 deals with the subdivision of categories and section 5 with the construction $C\mapsto\tilde C$. We prove that both $C$ and $\tilde C$ have the same homotopy type, and use the universal property to develop a lax version of Theorem A in section 6.

The final sections focus on the categorical path fibration. We introduce $\DC$ the category of simplices of a 2-category $C$ in section 7, and use the lax Theorem A to show that $\DC$ models the same homotopy type as $C$. In section 8 we define the path functor $L:\DC\laxto\cat$ and state and prove the main theorem. Some examples and a discussion on the necessity of the hypothesis are included. The last section relates our work with a classical result on delooping spaces that come from monoidal categories.

%Every strict monoidal category can be regarded as a 2-category with a single object. In this way theorem \ref{thm} can be thought of as a version of those for monoidal categories. It extends the range of applications to any 2-category, and provides a geometric organizational form for the theory.

\bigskip

\noindent{\bf Acknowledgments}

This is a version of part of my PhD Thesis at Universidad de Buenos Aires. I would like to express my gratitude to my advisor G. Minian, to CONICET for the finantial support, and also to E. Dubuc, M. Farinati, C. Casacuberta and F. Cukierman for their useful comments and suggestions.

I thank the referee who encouraged me to improve this work, and to IMPA - Rio de Janeiro where I made the revision. 

\section{Classifying spaces for small categories}
\label{cat}

This section summarizes Segal's classifying space for categories and its main features. We suggest \cite[$\S 1$]{quillen} as a general reference for this section.

\bigskip

The {\em nerve} $NC$ of a small category $C$ is the simplicial set whose $n$-simplices are the chains
$$c_0\to c_1\to\dots\to c_n$$
of $n$ composable arrows in $C$. Degeneracies in $NC$ insert an identity, the first and last faces drop an arrow, and the other faces compose two consecutive ones.
The {\em classifying space} $BC=|NC|$ is the geometric realization of the nerve. It is a CW-complex with one 0-cell for each object of $C$, one 1-cell for each arrow, one 2-cell for each commutative triangle, and so on. 

This way we have functors
$$N:\cat\to\sset \qquad B:\cat\to\top$$
where $\cat$, $\sset$ and $\top$ are, respectively, the categories of small categories, simplicial sets and topological spaces. These functors endow $\cat$ with homotopical notions: a map $u:C\to D$ in $\cat$ is a {\em weak equivalence} if it induces a homotopy equivalence between the classifying spaces, and a category $C$ is {\em contractible} if its classifying space is so.

It turns out that small categories are good models for homotopy types. More precisely, the nerve functor established an equivalence between the homotopy categories \cite[VI.3.3]{illusie}
$$\ho(\cat)\xto\sim\ho(\sset)\simeq \ho(\top)$$ 
Hence for every space $X$ there is a small category $C$ such that $X$ and $BC$ have the same weak homotopy type. Considering $C$ as a presentation of $X$, one seeks to compute the invariants of $X$ by using the structure of $C$.

Let us list some basic facts about the functor $B$:
\begin{itemize}\itemsep0pt
\item A natural transformation $u\then v:C\to D$ gives rise to a homotopy $Bu\simeq Bv:BC\to BD$.
\item If $u:C\to D$ admits an adjoint, then it is a weak equivalence.
\item If $C$ has initial or final object, then it is contractible.
\item There is a  homeomorphism $BC\cong BC\op$, where $C\op$ is the opposite of $C$.
\end{itemize}

The following fundamental tools originally appeared in \cite{quillen}. Recall that if $u:C\to D$ is a map in $\cat$ and $d\in\ob(D)$, the 
{\em homotopy fiber\footnote{also known as left fiber or comma category}}
$u/d$ is the category of pairs $(c,f)\in \ob(C)\times \fl(D)$ such that $f:u(c)\to d$. An arrow $(c,f)\to(c',f')$ in $u/d$ is given by an arrow $c\to c'$ of $C$ inducing a commutative triangle.

\begin{theorem}[Theorem A]
Given $u:C\to D$ a map in $\cat$, if $u/d$ is contractible for all $d$ then $u$ is a weak equivalence.
\end{theorem}

\begin{theorem}[Theorem B]\label{thmb}
Given $u:C\to D$, if every arrow $d\to d'$ in $D$ induces a weak equivalence $u/d\to u/d'$, then $B(u/d)$ has the homotopy type of the homotopy fiber of $BC\to BD$. In particular, there is a long exact sequence of homotopy groups relating those of $B(u/d)$, $BC$ and $BD$
\end{theorem}

Theorems A and B are especially useful when dealing with fibred categories.
Recall that a map $p:E\to B$ in $\cat$ is said to be a {\em prefibration\footnote{usually called pre-op-fibration or precofibration}} if the inclusion $p^{-1}(b)\to p/b$ admits a left adjoint for all $b$. 
Hence in a prefibration the actual fiber and the homotopy fiber have the same homotopy type. A {\em cleavage} for $p$ is a choice of such adjoint maps. In a prefibration endowed with a cleavage every arrow $b\to b'$ induces a {\em base-change functor} 
$$p^{-1}(b)\to p/b\to p/b'\to p^{-1}(b').$$
%and this way we have a {\em lax functor} $D\laxto\cat$ (see  example \ref{construction}).

The reader can find in \cite{dh2} a definition of fibred categories in terms of cartesian arrows and a general discussion on the subject, as well as a formulation of Theorems A and B within this framework.

%\begin{theorem}[Theorem A']
%A prefibration with contractible fibers is a weak equivalence.
%\end{theorem}
%
%\begin{theorem}[Theorem B']\label{thmb}
%Given $u:C\to D$, if every arrow $d\to d'$ in $D$ induces a weak equivalence $u/d\to u/d'$, then $B(u/d)$ has the homotopy type of the homotopy fiber of $BC\to BD$.
%\end{theorem}

\section{2-categories and their morphisms}\label{2categories}

We recall here some basic facts concerning 2-categories and lax functors, and fix some notations we shall use hereafter. We refer to \cite{borceux} for further details.

\bigskip

A 2-category $C$ is a category enriched over $\cat$. 
It consists of the following data:
a class of {\em objects} $C_0$; 
for each pair $c,c'\in C_0$ a (small) category of {\em arrows} $C(c,c')$; 
for each $c\in C_0$ an {\em identity arrow} $\id_c$ which is an object of $C(c,c)$; 
for each triple $c,c',c''\in C_0$ a {\em composition} functor
$\circ:C(c',c'')\times C(c,c')\to C(c,c'').$
These data must satisfy the usual neutral and associative axioms.
In a 2-category there are three levels of structure: objects; arrows between them; and {\em 2-cells}, which are the arrows between the arrows. The usual picture is
$$\xymatrix{c \rtwocell^f_g{\alpha} & c'}$$
As usual, we denote by $\beta\circ\alpha$ the {\em horizontal} composition, and by $\beta\bullet\alpha$ the {\em vertical} composition, say that of the categories $C(c,c')$.

\vspace{-15pt}
$$\xymatrix{c \rtwocell & c' \rtwocell & c''} \overset{\circ}{\leadsto}
\xymatrix{c \rtwocell & c''} \qquad
\xymatrix{c\ruppertwocell \rlowertwocell \rto & c'}\overset{\bullet}{\leadsto}
\xymatrix{c\rtwocell & c'}$$
\vspace{-15pt}

\begin{example}
The paradigmatic example of a 2-category is that of small categories, functors and natural transformations. Another basic example is that of spaces, continuous maps and (homotopy classes of) homotopies. 
\end{example}

A 2-category is said to be {\em small} if its objects form a set. Of course, the examples above are not small. Next we shall work with small 2-categories, associate topological spaces to them and study their homotopy types.

\begin{example}
Every (small) category $C$ can be regarded as a 2-category whose only 2-cells are identities.
\end{example}

\begin{example}\label{overline}
Every (small strict) monoidal category $(M,\tensor)$ can be regarded as a 2-category with a single object, one arrow for each object of $M$, one 2-cell for each arrow of $M$, and horizontal composition given by $\tensor$.
\end{example}

A {\em 2-functor} $u:C\to D$ between 2-categories consists of a map $u:C_0\to D_0$ together with functors $u:C(c,c')\to C(u(c),u(c'))$ such that all the structure is preserved. % These 2-functors are the simplest notion of morphisms for 2-categories. 
The 2-categories and 2-functors form a category, which we shall denote by $2\cat$.

\begin{example}\label{representable}
Let $C$ be a 2-category, $C\op$ its opposite category (described below) and $c$ an object. The 2-functor {\em represented by $c$} is denoted by $h^{c}:C\op\to\cat$ and defined as follows:
\begin{itemize}\itemsep0pt
\item[i)] $h^{c}(c')=C(c',c)$ for every object $c'$ of $C$;
\item[ii)] given $c',c''$ objects of $C$, the corresponding functor is %associated to the composition by the exponential law (in 2-cells 
%$h^{c}(\alpha)(\beta)=\alpha\circ\beta$).
$$h^{c}:C\op(c',c'')=C(c'',c')\to \cat(C(c',c),C(c'',c))
\qquad h^{c}(\alpha)(\beta)=\beta\circ\alpha$$
\end{itemize}
\end{example}

A 2-functor must preserve the structure {\it on the nose}. This implies identities between functors and is too restrictive. We can relax this condition by requiring the existence of natural transformations subject to {\em coherence axioms}. Experience has shown that these lax maps emerge naturally and are often useful.

%Given a 2-functor $u:C\to D$, there should be an identity of functors
%$$u\cdot\circ=\circ\cdot u:C(c,c')\times C(c',c'')\to D(u(c),u(c''))$$
%and this sometimes is too restrictive. One can relax this condition by requiring the existence of a natural isomorphism, or even just a natural transformation. These natural maps must satisfy some {\em coherence axioms}. Experience has shown that this lax morphisms emerge naturally and are often useful.

A {\em (normal) lax functor} between 2-categories $u:C\laxto D$ consists of a map $u:C_0\to D_0$; for each pair $c,c'\in C_0$ a functor $u:C(c,c')\to D(u(c),u(c'))$; and for each pair $f:c\to c'$, $g:c'\to c''$ in $C$ a {\em structural 2-cell} $u_{g,f}:u(gf)\then u(g)u(f) \in D_2$. The following axioms hold
\begin{enumerate}[i)]\itemsep 0pt
\item $u(\id_c)=\id_{u(c)}$ and $u_{\id,f}=\id=u_{f,\id}$ (normality);
\vspace{-5pt}
\item $(u(b)\circ u(a))\bullet u_{g,f} = u_{g',f'}\bullet u(b\circ a)$
for all $\xymatrix{ c \rtwocell^f_{f'}{a}& c'}$ and 
$\xymatrix{ c' \rtwocell^g_{g'}{b}& c''}$;
\vspace{-5pt}
\item
$(u_{h,g}\circ u(f))\bullet u_{hg,f}=
(u(h)\circ u_{g,f})\bullet u_{h,gf}$
for every chain $c\xto f c'\xto g c''\xto h c'''$.
\end{enumerate}

%Of course, a lax functor with trivial structural cells is just a 2-functor as defined above.
We shall denote the category of 2-categories and lax functors by $2\catl$.

\begin{example}\label{construction}
Given $p:E\to B$ a prefibration in $\cat$ endowed with a cleavage, we can define a lax functor $B\laxto\cat$ by giving to each object $b$ its fiber $p^{-1}(b)$ and to each arrow $b\to b'$ the corresponding base-change functor. The structural 2-cells are induced by the universal property of the adjoint.

Conversely, given a lax functor $F:B\laxto\cat$, its {\em Grothendieck construction} is the category $F\rtimes B$ of pairs $(x,b)$ such that $x\in\ob(F(b))$. An arrow $(x,b)\to(x',b')$ is a pair $(f,\alpha)$ such that $\alpha:b\to b'$ and $f:F(\alpha)x\to x'$. Composition is given by the rule
$(f',\alpha')\circ(f,\alpha)=
(f'F(\alpha')(f)F_{\alpha',\alpha}^x,\alpha'\alpha)$.
The projection $F\rtimes B\to B$ is a prefibration.

These constructions yield a 2-equivalence between lax functors and prefibrations with a cleavage \cite[Vol 2 \S 8]{borceux}.
\end{example}

%Given $B$ a category, a lax functor $F:B\laxto\cat$ is essentially the same as a prefibration $E\to B$ with a distinguished cleavage. More precisely, the {\em Grothendieck construction} of $F$ is the category $F\rtimes B$ of pairs $(x,b)$ such that $x\in\ob(F(b))$. An arrow $(x,b)\to(x',b')$ in $F\rtimes B$ is a pair $(f,\alpha)$ such that $\alpha:b\to b'$ and $f:F(\alpha)x\to x'$. Composition is given by the rule
%$(f',\alpha')\circ(f,\alpha)=
%(f'F(\alpha')F_{\alpha',\alpha}^x,\alpha'\alpha)$.
%It can be shown that the Grothendieck construction establishes a 2-equivalence between lax functors $B\laxto\cat$ and prefibrations $E\to B$ \cite[Vol 2 \S 8]{borceux}. 

%\bigskip

A classifying space does not care about the orientation of arrows and 2-cells. 
Keeping this in mind, it will be useful to recall the following constructions.

Given $C$ a 2-category, let $C\op$ be the one obtained by reversing the arrows, and let $C'$ be the one obtained by reversing the 2-cells.  $$C_0=C\op_0=C'_0 \qquad C\op(c,c')=C(c',c) \quad C'(c,c')=C(c,c')\op$$
%Composition and identities are defined in the obvious way.

%Given $C$ a 2-category, let $C\op$ be the 2-category obtained by reversing the orientation of the arrows. Thus $C\op$ has the same objects as $C$, and $C\op(c,c')=C(c',c)$. 
%In a similar fashion, we might construct $C'$ from $C$ by reversing the orientation of the 2-cells, namely $C'_0=C_0$ and $C'(x,y)=C(x,y)\op$.
%In both constructions composition and identities are defined in the obvious way.

%The following identities clearly hold: $C''=C$, $C\op{}\op=C$ and $C'{}\op=C\op{}'$.
The construction $C\mapsto C\op$ is functorial with respect to lax functors, whereas the construction $C\mapsto C'$ is functorial only with respect to 2-functors. Actually, a lax functor $u:C\laxto D$ induces an {\em oplax functor} $u':C'\to D'$. %, namely a sort of lax functor in which structural 2-cells go  backwards.

\section{Spaces associated to 2-categories}
\label{B_2C}

We recall here two ways in which 2-categories give rise to topological spaces. They both are extensions of the classifying space of a small category, and yield the same homotopy type. The reference for this section is \cite{bc}.

\bigskip

Given $C$ a 2-category, the {\em nerve} $\u NC$ is the simplicial category defined by
%We have seen that the classifying space of a small category $C$ is the geometric realization of its nerve $NC$, which can be defined by
$$\u NC_n=\coprod_{c_0,\dots,c_n} C(c_0,c_1)\times\dots\times C(c_{n-1},c_n)$$
%with products and coproducts in $\cat$. 
By applying the nerve functor $\cat\to\sset$ in each degree we get a bisimplicial set, which we call the {\em 2-nerve} and denote by $N_2C$.

For instance, a simplex $s$ in $(N_2C)_{2,3}$ is just a diagram as follows:
$$\xymatrix{
c_0 \ar[r]\ar@{=}[d] \ar@{}[dr]|{\hspace{10pt}\Downarrow\alpha_{1,1}}&
c_1 \ar[r]\ar@{=}[d] \ar@{}[dr]|{\hspace{10pt}\Downarrow\alpha_{1,2}}&
c_2 \ar[r]\ar@{=}[d] \ar@{}[dr]|{\hspace{10pt}\Downarrow\alpha_{1,3}}&
c_3 \ar@{=}[d] \\
c_0 \ar[r]\ar@{=}[d] \ar@{}[dr]|{\hspace{10pt}\Downarrow\alpha_{2,1}}&
c_1 \ar[r]\ar@{=}[d] \ar@{}[dr]|{\hspace{10pt}\Downarrow\alpha_{2,2}}&
c_2 \ar[r]\ar@{=}[d] \ar@{}[dr]|{\hspace{10pt}\Downarrow\alpha_{2,3}}&
c_3 \ar@{=}[d] \\
c_0 \ar[r] & c_1 \ar[r] & c_2 \ar[r] & c_3 }$$

%\begin{remark}
%Let $\theta$ be the 2-category defined by ... . From this construction we have a functor $\theta:\Delta\times\Delta\to2\cat$.
%There is a natural isomorphism
%$$N_2(C)_{m,n}=2\cat(\theta_{m,n},C)$$
%between the 2-nerve and the functor induced by $\theta$.
%\end{remark}

The {\em 2-classifying space} $B_2C$ of $C$ is the geometric realization of the 2-nerve $|\diag(N_2C)|$. This space can also be obtained by first realizing in one direction and then in the other.

\begin{example}
If $C$ is a category regarded as a 2-category in the usual way, then $B_2C$ is homeomorphic to $BC$.
\end{example}

\begin{example}
If $M$ is a monoidal category then $B_2 M$ is the classifying space of the topological monoid $BM$.
\end{example}

Let $u:C\to D$ be a 2-functor. We shall say that $u$ is a {\em weak equivalence} if it induces a homotopy equivalence $B_2C\to B_2D$. We shall say that $u$ is a {\em local weak equivalence} if $u_*:C(c,c')\to D(u(c),u(c'))$ is a weak equivalence in $\cat$ for all $c,c'\in C_0$. A 2-category $C$ is said to be {\em contractible} if $B_2C$ is so.

\begin{proposition}\label{localweak}
If a 2-functor $u:C\to D$ is a local weak equivalence and induces a bijection between the objects then $u$ is a weak equivalence.
\end{proposition}

\begin{proof}
It is well-known that a map of bisimplicial sets that is a weak equivalence at each level yields a weak equivalence \cite[IV-1.9]{gj}. The result follows by observing that $(N_2C)_{\ast,n}\to (N_2D)_{\ast,n}$ is a weak equivalence for each $n$.
\end{proof}

%\begin{observacion}
%De la naturalidad de los homeomorfismos $B_2C\cong B_2C'\cong B_2C\op$ sigue que las siguientes afirmaciones son equivalentes:
%\begin{itemize}
%\item $u:C\to D$ es una equivalencia homotópica;
%\item $u':C'\to D'$ es una equivalencia homotópica;
%\item $u\op:C\op\to D\op$ es una equivalencia homotópica.
%\end{itemize}
%\end{observacion}

There is another natural way to associate a topological space to a 2-category $C$. It is constructed by means of the geometric nerve, following the terminology of \cite{bc}.
%Let $g$ be the covariant functor obtained as the composition of the following inclusions.
%$$\Delta\to\cat\to2\cat\to2\catl$$
The {\em geometric nerve} $N_gC$ is the simplicial set given by
$$(N_gC)_n=2\catl(\ord n,C)$$
where $\ord n=\{0\to 1\to\dots\to n\}$ is viewed as a 2-category with trivial 2-cells.
Its 0-simplices are the objects of $C$, its 1-simplices are the arrows of $C$, its 2-simplices are diagrams of the form
$$\xymatrix{ x_0 \ar[dr] \ar[rr] &\ar@{}[d]|(.4){\Downarrow} & x_2\\& x_1\ar[ru] }$$
and its simplices of higher dimension are completely determined by these, namely $N_gC$ is 2-coskeletal (cf. \cite{street}).
We denote by $B_gC$ the geometric realization of the geometric nerve.
The geometric nerve is easier to define, and it manages to describe completely the structure of $C$. Despite that, it is hard to make it explicit even in very simple examples.

\begin{theorem}[\cite{bc}]
There is a natural homotopy equivalence $B_2C\simeq B_g C$.
\end{theorem}

%\begin{proof}[Sketch of proof]
%They construct a bisimplicial set $K$ and inclusions
%$$N_gC \to K \from N_2C$$
%where $N_gC$ is considered as a bisimplicial set constant in one direction.
%
%...
%\end{proof}

By means of the previous equivalence, a lax functor $u:C\laxto D$ gives a map $B_2C\to B_2D$ well-defined up to homotopy, so it does make sense to say that such a map is a weak equivalence. %This definition extends the one given for 2-functors.

\begin{remark}\label{turning}
There are canonical natural homeomorphisms
$$B_2C\cong B_2C\op \qquad B_2C\cong B_2C'$$
Given a 2-functor $u$, it follows that whenever $u$, $u\op$ or $u'$ is a weak equivalence, then so are the others. 
\end{remark}

%\section{A categorical construction}
%
%{\it Subdivision revisited. Lax functors. Triangle operation. The construction. Theorem concerning universal property of the construction.}
%
%In this section we associate a 2-category $\tilde C$ to every category $C$ in a way such that the following universal property holds:
%a (normal) lax functor $C\laxto D$ is the same as a 2-functor $\tilde C\to D$. The  objects of $\tilde C$ are that of $C$, and roughly speaking an arrow in $\tilde C$ is a chain of arrows in $C$; the 2-cells of $\tilde C$ are the ways to obtain one chain from another. To set a precise definitions we shall need the notion of subdivision of categories. We recall it here, and give a new more effective definition, then we construct $\tilde C$ and finally we prove the universal property.

\section{Subdivision revisited}\label{subdiv}

The subdivision of categories is similar to the barycentric subdivision of polyhedra.
It is a functor $\sd:\cat\to\cat$ that assigns to every category $C$ another $\sd(C)$ with the same homotopy type and in some sense locally simpler. 
This construction appears in early works \cite{anderson,dh1}.
Here we present a new characterization of $\sd(C)$ that brings some clarification and makes proofs easier.

\bigskip

Let $C$ be a small category and let $\Delta/C$ be the {\em category of simplices} of $C$. The objects of $\Delta/C$ are the simplices of $NC$, say functors $x:[n]\to C$, and the arrows $a_*:x\to x'$ are given by ordinal maps $a:[n]\to[n']$ such that $x'\cdot a= x$. By mapping a chain $x$ to its last object $x(n)$ one gets a functor $\sup:\Delta/C\to C$. It is well-known that this is a weak equivalence (cf. \cite{illusie}-VI.3.3).

\begin{definition}\label{newdefsd}
We define a relation $\sim$ on the arrows of $\Delta/C$ by the rule
$$a_*\sim b_*:x\to x' \iff x'(m(i)\to M(i))=\id \ \forall i$$
where $m(i)=\min\{a(i),b(i)\}$ and $M(i)=\max\{a(i),b(i)\}$.
\end{definition}

It is routine to verify that $\sim$ is an equivalence relation compatible with the composition. We denote by $[\Delta/C]$ the quotient category, whose objects are those of $\Delta/C$ and whose arrows are the classes under $\sim$.
The {\em subdivision} $\sd(C)$ is the full subcategory of $[\Delta/C]$ formed by the non-degenerate simplices.

The functor $\sup:\Delta/C\to C$ clearly induces another one $[\Delta/C]\to C$. Call $\eps:\sd(C)\to C$ its restriction. Analogously, $\eps':\sd(C)\to C\op$ is defined by using $\inf$ instead of $\sup$.
The main features concerning subdivision are

\begin{itemize}\itemsep0pt
\item The construction $\sd$ actually defines a functor $\cat\to\cat$.
\item The map $\eps:\sd(C)\to C$ is a weak equivalence for all $C$.
\item $\sd^2(C)$ is a poset for all $C$.
\end{itemize}

%Whence we might say that $\sd(C)$ is an alternative model for the homotopy type of $BC$, locally simpler than $C$.
Here $\sd(C)$ plays the role of a functorial {\em resolution} or {\em cofibrant replacement} of $C$, and $\eps$ that of the {\em augmentation map}.
Proofs and further details concerning subdivision can be consulted in \cite{anderson,dh1}. The equivalence between the constructions given there and that  of definition \ref{newdefsd} easily follows from

\begin{proposition}
$\sim$ is generated by the following elementary relation:
$$a_*\approx b_* \iff \exists i_0 \text{ such that }
\begin{cases}
a(i)=b(i) & i\neq i_0\\
x'(a(i_0)\to b(i_0))=\id\end{cases}$$
\end{proposition}
\begin{proof}
Clearly $\sim$ is an equivalence relation containing $\approx$. Let us prove that it is the smallest with this property.

Suppose first that $a_*\sim b_*:x\to x'$, with $a,b:[n]\to[n']$
such that $a(i)\leq b(i)$ for all $i$. Then, if $c_k:[n]\to[n']$ is given by
$$c_k(i)=\begin{cases}a(i) & i< k \\ b(i) & i\geq k\end{cases}$$
we have that $(c_k)_*:x\to x'$ is an arrow in $\Delta/C$ for all $k$, $(c_{n+1})_*=a_*$, $(c_0)_*=b_*$ and $(c_k)_*\approx(c_{k-1})_*$.

Now let $a_*\sim b_*: x\to x'$ be any two equivalent maps. We construct $m,M:[n]\to [n']$ by $m(i)=\min(a(i),b(i))$ and $M(i)=\max(a(i),b(i))$. Clearly 
$$m_*\sim a_*\sim b_*\sim M_*$$
and because of the first case we analysed $m_*$ and $a_*$  are related by a chain of elementary steps, as well as $m_*$ and $b_*$.
\end{proof}

\section{The construction $C\mapsto\tilde C$}
\label{tilde}

Given $C$ a small category, we construct here a 2-category $\tilde C$ and a lax functor $\eta:C\laxto \tilde C$ with the following universal property:
$$\xymatrix{C \ar@{-->}[r]^\eta \ar@{-->}[rd]_{\forall v}&\tilde C \ar[d]^{\exists! u}\\&D }
\qquad
\text{\parbox[t]{6cm}{for every lax functor $v:C\laxto D$ \\ there exists a unique 2-functor\\ $u:\tilde C\to D$ such that $u\eta=v$.}}$$
$\tilde C$ has the same objects as $C$, its arrows are chains of composable arrows and its 2-cells are ways to obtain one chain from another. In order to give a precise definition we shall make use of the categorical subdivision.

In \cite{cordier,gray2} the same problem is considered, without demanding normality. We carry out the construction in detail for we understand that the references are rather vague and contain some mistakes. The result is often attributed to Jean Benabou.
% However, we think that the construction presented there contains some unsolved technical problems.

\bigskip

From here on $C$ will denote a small category. Given $c,c'$ objects of $C$, let $\tilde C(c,c')$ be the fiber of the functor $\eps'\times\eps:\sd(C)\to C\op\times C$ over the object $(c,c')$. 
In other words, $\tilde C(c,c')$ is the non-full subcategory of $\sd(C)$ formed by the chains $x:\ord{n}\to C$ that start at $c$ and end at $c'$, and the maps $[a_*]$ which preserve the first and the last element.

Assume $c\neq c'$. If $[a_*]:x\to x'$ is an arrow in $\tilde C(c,c')$ then $x'(0\to a(0))=\id_c$ and $x'(a(n)\to n')=\id_{c'}$. It follows that $a_*\sim a'_*$ where
$$a':\ord{n}\to\ord{n'} \qquad
a'(i)=\begin{cases} 0 & i=0\\ a(i) & 0<i<n \\ n' & i=n \end{cases}$$
We conclude that every arrow $[a_*]:x\to x'$ in $\tilde C(c,c')$ can be represented by an injective order map $a$ that preserves the first and the last element. Thus we have

\begin{proposition}\label{onearrow}
If $n=1$, then there is at most one arrow $x\to x'$ in $\tilde C(c,c')$.
\end{proposition}

It easily follows that each component of $\tilde C(c,c')$ has an initial element, which is given by a 1-simplex.

When $c=c'$ the structure of $\tilde C(c,c)$ is quite similar, except that there is a special component with initial element given by the 0-simplex $c$.

\smallskip

Given $a:\ord{p}\to\ord{p'}$ and $b:\ord q\to\ord{q'}$, we define $a\triangleleft b:\ord{p+q}\to \ord{p'+q'}$ as the map
$$(a\triangleleft b)(i)=\begin{cases}a(i) & i\leq p\\ b(i-p)+p' & i>p\end{cases}$$
We emphasize the asymmetry of this definition: the last value of $a$ is kept and the first of $b$ is dropped. This is arbitrary and other variants also work.

\begin{proposition}\label{triangle}
The following hold:
\begin{enumerate}[\hspace{15pt}a)]\itemsep0pt
\item $\id_{\ord p}\triangleleft \id_{\ord q}=\id_{\ord{p+q}}$;
\item $(a'\triangleleft b')\circ(a\triangleleft b)=(a'\circ a)\triangleleft (b'\circ b)$ if $b(i)>0$ whenever $i>0$;
\item $(a\triangleleft b)\triangleleft c= a\triangleleft(b\triangleleft c)$;
\item $a\triangleleft \id_{\ord 0}= a$; $\id_{\ord 0}\triangleleft a=a$ if $a(0)=0$.
\end{enumerate}
\end{proposition}

Given $c,c',c''$ objects in $C$, we define the juxtaposition functor
$$\odot:\tilde C(c,c')\times \tilde C(c',c'')\to \tilde C(c,c'')$$
as follows.
If $x$ and $x'$ are objects of $\tilde C(c,c')$ and $\tilde C(c',c'')$ respectively, then $x\odot x':\ord{n+n'}\to C$ is the chain of arrows given by juxtaposition, say
$$(x\odot x')(i-1\to i)=\begin{cases}
x(i-1\to i) & i \leq n\\
x'(i-1-n\to i-n) & i > n
\end{cases}$$
In arrows, $[a_*]\odot[b_*]$ is defined as $[(a\triangleleft b)_*]$.
This definition does not depend on the representative, for
if $a_*\sim a'_*$ and  $b_*\sim b'_*$ then $(a\triangleleft b)_*\sim (a'\triangleleft b')_*$. This is immediate from our description of the arrows of the subdivision (cf. definition \ref{newdefsd}).
Clearly $\odot$ preserves identity elements. For composition, the identity
$$([a'_*]\odot[b'_*])\circ([a_*]\odot[b_*])=
([a'_*]\circ[a_*])\odot([b'_*]\circ[b_*])$$
holds because of the way arrows in $\sd(C)$ are composed, proposition \ref{triangle} and the fact that every arrow in $\sd(C)$ can be represented by an injective order map.

\begin{definition}
Given $C\in\cat$, we define $\tilde C$ as the 2-category such that:
\begin{enumerate}[\hspace{15pt}i)]\itemsep0pt
\item its objects are those of $C$, i.e. $\tilde C_0=C_0$;
\item for each pair $c,c'\in \tilde C_0$ the category $\tilde C(c,c')$ is defined as above;
\item the identity $\id_c\in \tilde C(c,c)$ is the 0-simplex of $NC$ induced by $c$;
\item the composition $\tilde C(c,c')\times \tilde C(c',c'')\to \tilde C(c,c'')$ is the juxtaposition $\odot$.
\end{enumerate}
\end{definition}
One checks the associative and unit axioms for $\odot$ by using proposition \ref{triangle}.

Let $\eta:C\laxto\tilde C$ be the lax functor which is the identity on objects, maps a non-trivial arrow $f\in C(c,c')$ to the 1-simplex of $NC$ induced by $f$, and carries identities into identities. There is only one way to define the structural 2-cells (cf. \ref{onearrow}), and this makes all the axioms trivially hold.

Given $D$ a 2-category and $u:\tilde C\to D$ a 2-functor, the composition $u\eta$ is a lax functor $u\eta:C\laxto D$. This way we have a map
$$2\cat(\tilde C,D)\to 2\catl(C,D)$$

\begin{theorem}\label{up}
The map above is a bijection.
\end{theorem}

In other words, the universal property of $\tilde C$ stated at the beginning of the section holds.

\begin{proof}
In order to prove that the map is injective, we need to show that we can recover $u$ from $v=u\eta$. This is clear on objects, for $\eta:C_0\to\tilde C_0$ is the identity. Besides, every arrow in $\tilde C$ is a chain $x:[n]\to C$ and can be written as the composition (juxtaposition) of elementary arrows $[1]\to C$. Since these arrows belong to the image of $\eta$, we conclude that the behaviour of $u$ on arrows is settled by $v$. Now consider an {\em elementary 2-cell}
$$[a_*]:x\then x':c\to c'\qquad n\leq 1, n'=2$$
The structural 2-cells of $v$ are obtained from those of $\eta$, say $v_{g,f}=u(\eta_{g,f})$. It follows that
$u([a_*])=u(\eta_{x'(1\to 2),x'(0\to 1)})=v_{x'(1\to 2),x'(0\to 1)}$. 
Since every 2-cell of $\tilde C$ can be obtained from the elementary ones
by using $\circ$ and $\bullet$, the injectivity follows.

For the surjective part, we must show that every lax functor $v:C\laxto D$ equals $u\eta$ for some $u$. Given $v$, we construct $u$ as follows:
\begin{itemize}\itemsep0pt
\item let $u(c)=v(c)$ for every object $c$;
\item given $x\in\tilde C(c,c')_0$ let $u(x)=v(x(n-1\to n))\dots v(x(0\to 1))$;
\item given $[a_*]:x\then x':c\to c'$, with $a:\ord{n}\to\ord{n'}$, let
$u([a_*])=\alpha_{n}\circ\dots\circ\alpha_1$
where 
$$\alpha_i:vx(i-1\to i)\then vx'(a(i)-1\to a(i))\dots vx'(a(i-1)\to a(i-1)+1)$$
is the 2-cell induced by $u$.
\end{itemize}
It is straightforward to check that $u$ is a 2-functor and that $u\eta=v$.
\end{proof}

The lax functor $\eta$ has a left inverse, say $\pi:\tilde C\to C$, which is a 2-functor. Of course, $\pi$ is the identity in the objects. Given $c,c'\in C_0$, the functor $\pi:\tilde C(c,c')\to C(c,c')$ maps a chain $x$ to its total composition $x(0\to n)$. This 2-functor $\pi$ is related to $\id_C$ via the universal property of $\tilde C$.

\begin{proposition}\label{etaisweq}
The map $\eta:C\laxto \tilde C$ is a weak equivalence.
\end{proposition}

\begin{proof}
It suffices to show that its left inverse $\pi:\tilde C\to C$ is so. 
If $f:c\to c'$ is an arrow in $C$ and $x\in \tilde C(c,c')$, then there is at most one arrow $\eta(f)\to x$ in $\tilde C(c,c')$ (see \ref{onearrow}), and it exists iff $\pi(x)=f$. Thus we have an adjunction $\eta\dashv\pi$ between $\tilde C(c,c')$ and the discrete category $C(c,c')$, that is, every component of $\tilde C(c,c')$ has an initial element. The functor $\tilde C(c,c')\to C(c,c')$ is a weak equivalence because it admits an adjoint. Thus $\pi$ is a local weak equivalence and a bijection on the objects. The proof ends by applying proposition \ref{localweak}.
\end{proof}

\section{Theorem A for 2-functors and lax functors}

Quillen's Theorem A asserts that a map in $\cat$ is a weak equivalence if its homotopy fibers are contractible. This theorem has recently been extended to 2-categories and 2-functors \cite{bc}. Here we introduce the homotopy fiber of a lax functor, and use the results from the previous sections to establish a version of Theorem A for lax functors.

\bigskip

Given $C$ a 2-category we denote by $C_0$, $C_1$ and $C_2$ its sets of objects, arrows and 2-cells respectively.

\begin{definition}\label{homfiber}(cf. \cite{gray}) 
Given $u:C\to D$ a 2-functor, and given $d$ an object of $D$, the {\em homotopy fiber} of $u$ over $d$ is the 2-category $u/\!/d$ defined as follows. Its objects are pairs $(c,\phi)\in C_0\times D_1$ such that $\phi:u(c)\to d$.
$$\xymatrix@R=10pt{(c,\phi) & & u(c)\ar[rr]^\phi & & d}$$
Its arrows $(c,\phi)\to(c',\phi')$ are pairs $(f,\alpha)\in C_1\times D_2$ such that $f:c\to c'$ and $\alpha:\phi'u(f)\then \phi$
$$\xymatrix@R=10pt{
(c,\phi) \ar[dd]_{(f,\alpha)}& & u(c)\ar[dd]_{u(f)}\ar[rrd]^\phi & \\
 & & & & d \\
(c',\phi') & & u(c') \ar[urr]_{\phi'} &  \ar@{}[uul]|{\Uparrow\alpha}}$$
Its 2-cells $\beta:(f,\alpha)\then(f',\alpha')$ are given by 2-cells $\beta:f\then f'$ of $C$ such that
$\alpha'\bullet(\phi'\circ u(\beta))= \alpha$;
$$\xymatrix@R=10pt{
(c,\phi) \ar@/_1.1pc/[dd]_{(f,\alpha)} \ar@/^1pc/[dd]^{(f',\alpha')}
& & u(c)\ar@/_1.1pc/[dd]_{u(f)} \ar@/^1pc/[dd]|{\ \ \ \ u(f')}\ar[rrd]^\phi & & \\
\underset{\beta}{\then} & & \underset{u(\beta)}{\then}& & d \\
(c',\phi') & & u(c') \ar[urr]_{\phi'} & \ar@{}[uu]|{\Uparrow\alpha'}}$$
The composition $\circ$ is given by $(g,\beta)\circ(f,\alpha)=(gf,\alpha\bullet(\beta\circ u(f)))$.
The composition $\bullet$ is that of $C$.
\end{definition}

%By an abuse of notation, we write $C/\!/d$ instead of $u/\!/d$.
Note that this extends the construction of homotopy fibers for functors that was recalled in section \ref{cat}. The following result is due to Bullejos and Cegarra \cite{bc}.

\begin{theorem}[Theorem A for 2-functors]\label{thma2fun}
Let $u:C\to D$ be a 2-functor. If the category $u/\!/d$ is contractible for all $d\in D_0$ then $u$ is a weak equivalence.
\end{theorem}

In view of remark \ref{turning}, this theorem admits many alternative formulations: if the homotopy fibers of $u$, $u'$ or $u\op$ are contractible, then $u$ is a weak equivalence. The version stated here is related with the original in \cite{bc} by the isomorphism of 2-categories
${(d/\!/u)\op}' \cong ({u\op}')/\!/d$, where $d/\!/u$ is the right fiber as presented in \cite{bc}.

%The homotopy fiber of a 2-functor admits the following variant. If $u:C\laxto D$ and $d\in D_0$, then the category $(C'/\!/d)'$ is quite similar to $C/\!/d$: they have the same objects, but an arrow $(c,\phi)\to(c',\phi')$ in $(C'/\!/d)'$ is a pair $(f,\alpha)\in C_1\times D_2$ such that $f:c\to c'$ and $\alpha:\phi\then \phi'u(f)$.
%$$\xymatrix@R=10pt{
%(c,\phi) \ar[dd]_{(f,\alpha)}&  u(c)\ar[dd]_{u(f)}\ar[rd]^\phi & \\ & & d \\
%(c',\phi') & u(c') \ar[ur]_{\phi'} & \ar@{}[uul]|{\Downarrow\alpha}}$$
%Therefore, in order to prove that a 2-functor is a weak equivalence it suffices to check that the categories $(C'/\!/d)'$ are contractible.

We extend the construction of homotopy fibers from 2-functors to lax functors in the following quite reasonable way.

\begin{definition}\label{homfiber2}
Given $u:C\laxto D$ and $d\in D_0$, we define the {\em homotopy fiber} $u/\!/d$ of $u$ over $d$ as the 2-category with objects, arrows and 2-cells as in definition \ref{homfiber}, but with horizontal composition given by:
$$(g,\beta)\circ(f,\alpha)=(gf,\alpha\bullet(\beta\circ u(f))\bullet(\phi''\circ u_{g,f}))
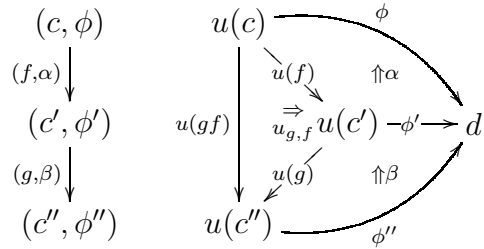
$$
$$\xymatrix@R=20pt@C=10pt{
(c,\phi) \ar[d]_{(f,\alpha)} & &  u(c)\ar[dd]_{u(gf)}
\ar[dr]|{u(f)} \ar@/^1.5pc/[rrrd]^\phi & & \\
(c',\phi') \ar[d]_{(g,\beta)} & &  & u(c') \ar[rr]|{\phi'} \ar[dl]|{u(g)}
\ar@{}[l]|{\substack{\Rightarrow\\ u_{g,f}}}
\ar@{}[ur]|{\Uparrow\alpha}
\ar@{}[dr]|{\Uparrow\beta}
& & d \\
(c'',\phi'')& & u(c'') \ar@/_1.5pc/[urrr]_{\phi''} & &
}$$
\end{definition}

%As usual, we shall write $C/\!/d$ instead of $u/\!/d$ when there is no place to confusion.
Note that when $C$ is a category, i.e. it has only trivial 2-cells, the homotopy fiber $u/\!/d$ is also a category.

\begin{theorem}[Theorem A for lax functors]\label{laxthma}
Let $C$ be a category, $D$ a 2-category and consider a lax functor $v:C\laxto D$. If $v/\!/d$ is contractible for all $d\in D_0$ then $v$ is a weak equivalence.
\end{theorem}

\begin{proof}
We use the universal property of $\tilde C$ to factor $v$ as $u\eta$.
$$\xymatrix{C \ar@{-->}[r]^\eta \ar@{-->}[rd]_{v}&\tilde C \ar[d]^{u}\\&D }$$
We have seen that $\eta$ is a weak equivalence (proposition \ref{etaisweq}), so it remains to prove that the 2-functor $u:\tilde C\to D$ is so.
Using Theorem A for 2-functors (theorem \ref{thma2fun}), we only have to check that $u/\!/d$ is contractible for every $d\in D_0$. By hypothesis we know that $v/\!/d$ is contractible. The maps $\eta:C\laxto\tilde C$ and $\pi:\tilde C\to C$ (proposition \ref{etaisweq}) induce morphisms between the homotopy fibers, say 
$\eta/\!/d:v/\!/d \to u/\!/d$ and $\pi/\!/d:u/\!/d \to v/\!/d$.
The same argument used in proposition \ref{etaisweq} shows that they establish a weak homotopy equivalence and the theorem follows.
\end{proof}

Bullejos and Cegarra extend Quillen's Theorem A to strict 2-functors between 2-categories. Our theorem \ref{laxthma} 
%includes lax functors into the range of application
extends it to lax functors, but we require $C$ to be a category (trivial 2-cells). Though this formulation is enough for our purpose, we believe that a stronger formulation holds, namely a Theorem A for lax functors between any 2-categories. A proof of this might follow the same lines as above by constructing $\tilde C$ for any 2-category.

\section{The category of simplices of a 2-category $C$}

\noindent
In this section we introduce the category of simplices $\DC$ of a 2-category $C$, extending the more familiar notion of $\Delta/C$ for categories. We shall use $\DC$ to prove theorem \ref{main} in the next section. A variation of $\DC$ was used in \cite{tesis} to prove that the homotopy categories of $\cat$ and $2\catl$ coincide.

\bigskip

If $C$ is a small category, then its {\em category of simplices} $\Delta/C$ can be presented in many conceptual and equivalent ways:

\begin{itemize}\itemsep0pt
\item $\Delta/C$ is the category of simplices of the simplicial set $NC$;
\item $\Delta/C$ is the homotopy fiber over $C$ of the inclusion $\Delta\to\cat$;
\item $\Delta/C$ is the opposite category to the Grothendieck construction over the (discrete) map $NC:\Delta\op\to\cat$.
\end{itemize}

When moving to 2-categories, these three constructions lead to different definitions. We shall adopt the last one.

\begin{definition}
Given $C$ a small 2-category, we define its {\em category of simplices} $\DC$ as $({\u NC}\rtimes{\Delta\op})\op$, namely the opposite category to the Grothendieck construction over the nerve functor $\u NC:\Delta\op\to\cat$ (cf. example \ref{construction}).
\end{definition}

The objects of $\DC$ are pairs $(n,x)$ such that $x:[n]\to C$ is a functor. 
In other words, they are the simplices of the nerve of the underlying category.
We visualize them as chains of composable arrows
$$x_0 \to x_1 \to \dots \to x_n$$
An arrow $(n,x)\to(n',x')$ in $\DC$ is a pair $(a,\alpha)$ with $a:\ord{n}\to\ord{n'}$ and $\alpha=(\alpha_1,\dots,\alpha_n)$ such that $\alpha_i:x'(a(i-1)\to a(i))\then x(i-1\to i)$ is a 2-cell of $C$.
We visualize an arrow as a 2-diagram in $C$ of the form:
$$\xymatrix@C=15pt{
 & & x_0 \ar@{=}[d] \ar[rr]&  &  x_1 \ar@{=}[d] \ar[rr] & & x_2 \ar@{=}[d] \ar[r] & \dots  \ar@{}[d]|\dots \ar[r] & x_{n} \ar@{=}[d] & & \\
x'_0 \ar[r]& \dots \ar[r] & x'_{a(0)} \ar[r] & \dots \ar@{}[u]|{\Uparrow} \ar[r] & x'_{a(1)} \ar[r] &  \dots \ar@{}[u]|{\Uparrow} \ar[r] & x'_{a(2)} \ar[r] & \dots  \ar[r]& x'_{a(n)}\ar[r] & \dots \ar[r] & x'_{n'}
}$$
Note that necessarily $x'(a(i))=x(i)$.
If $C$ is a category viewed as a 2-category in the usual way, then $\DC$ equals the usual category of simplices, so our definition is an extension indeed.

\begin{definition}
We define $\sup:\DC\laxto C$ as the following lax functor:
\begin{enumerate}\itemsep0pt
\item[i)] on objects, $\sup(n,x)=x_{n}$;
\item[ii)] if $(a,\alpha):(n,x)\to(n',x')$ is an arrow in $\DC$, then $$\sup(a,\alpha)=x'(a(n)\to n');$$
\item[iii)] given $(a,\alpha):(n,x)\to(n',x')$ and $(b,\beta):(n',x')\to(n'',x'')$, we have 
\begin{align*}
\sup((b,\beta)\circ(a,\alpha))&=x''(ba(n)\to n''),\\
\sup(b,\beta)\circ\sup(a,\alpha)&=x''(b(n')\to n'')\circ x'(a(n)\to n').
\end{align*}
and the structural 2-cell $\sup_{(b,\beta),(a,\alpha)}$ is defined by 
\begin{align*}
{\sup} _{(b,\beta),(a,\alpha)}=
x''(b(n')\to n'')\circ\beta_{n'}\circ\beta_{n'-1}\circ\dots\circ\beta_{a(n)+1}
\end{align*}
%$$\xymatrix@C=15pt{
%x_0 \ar[r]& ... \ar[r] & x_{\xi 0} \ar[r] & ... \ar@{}[d]|{\Downarrow} \ar[r] & x_{\xi 1} \ar[r] &  ... \ar@{}[d]|{\Downarrow} \ar[r] & x_{\xi2} \ar[r] & ... & ... \ar[r]& x_{\xi\eta 0}\\
% & & y_0 \ar@{=}[u] \ar[rr]&  &  y_1 \ar@{=}[u] \ar[rr] & & y_2 \ar@{=}[u] \ar[r] & ... & ...  \ar[r] & y_{\eta0} \ar@{=}[u] \\ & & & & & & & & & z_0 \ar@{=}[u]}$$
\end{enumerate}
\end{definition}

It is routine to check that $\sup$ is actually a lax functor, namely that it satisfies axioms i, ii and iii for lax functors (section \ref{2categories}).

\begin{theorem}\label{2suplaxo}
The map $\sup:\DC\laxto C$ is a weak equivalence. Thus $C$ and $\DC$ model the same homotopy type.
\end{theorem}

\begin{proof}
In view of our version \ref{laxthma} of Quillen's Theorem A, we need to verify  that the homotopy fibers $\sup/\!/c$ as defined in \ref{homfiber2} are contractible.

Fix $c$ an object of $C$, and let $i:(\DC)_c\to \sup/\!/c$ be the inclusion of the fiber into the homotopy fiber -- this is a map in $\cat$. We define a map $r:\sup/\!/c\to(\DC)_c$ and natural transformations 
$\eta:\id_{\sup/\!/c}\then ir$ and $\eps:\cte\then ir$, where $\cte(x,f)=(c,\id_c)$ is the constant functor. Because a natural transformation gives rise to a homotopy when taking classifying spaces, it follows that the identity of $\sup/\!/c$ is homotopic to a constant and hence it is contractible.

Given $(x,f)$ an object of $\sup/\!/c$, we define $r(x,f)$ as the simplex of $NC$ obtained by extending $x$ with $f$, that is
$$r(x,f)=(x_0\to x_1\to \dots\to x_n\xto f c)$$
An arrow $(x,f)\to(x',f')$ in $\sup/\!/c$ is a triple $((a,\alpha),\beta)$
$$\xymatrix@C=10pt{
 & & x_0 \ar@{=}[d] \ar[rr]&  &  x_1 \ar@{=}[d] \ar[rr] & & x_2 \ar@{=}[d] \ar[r] & \dots   \ar[r] & x_n \ar@{=}[d] \ar[rrr]^f & & & c \ar@{=}[d] \\
x'_0 \ar[r]& \dots \ar[r] & x'_{a(0)} \ar[r] & \dots \ar@{}[u]|{\Uparrow\alpha_1} \ar[r] & x'_{a(1)} \ar[r] &  \dots \ar@{}[u]|{\Uparrow\alpha_2} \ar[r] & x'_{a(2)} \ar[r] & \dots \ar@{}[u]|{\dots} \ar[r]& x'_{a(n)}\ar[r] & \dots \ar@{}[ur]|{\Uparrow\beta}\ar[r] & x'_{n'} \ar[r]_{f'}& c
}$$
We define $r((a,\alpha),\beta):r(x,f)\to r(x',f')$ as the map $(a',\alpha')$ with
$$a':\ord{n+1}\to\ord{n'+1} \qquad a'(j)=
\begin{cases}a(j) & j\leqslant n\\ n'+1 & j=n+1\end{cases}$$
and
$$\alpha'_j=\begin{cases}\alpha_j & j\leq n\\ \beta & j=n+1\end{cases}$$
Finally, the natural map $\eta$ is given by $[n]\to[n+1]$, $i\mapsto i$, and the natural map $\eps$ is given by $[0]\to[n+1]$, $0\mapsto n+1$.
\end{proof}

The lax functor $\inf:(\DC)\laxto C\op$ is defined analogously, and it happens to be a weak equivalence, too. Moreover, under the obvious isomorphism $\DC\cong\Delta/\!/(C\op)$, the functor $\inf_C$ can be identified with $\sup_{C\op}$.
$$\xymatrix{\DC \ar@{-->}[]+<-25pt,-10pt>;[d]_{\inf_C}\cong \Delta/\!/(C\op)
\ar@{-->}[]+<15pt,-10pt>;[d]^{\sup_{C\op}}\\ C\op}$$

\section{The loop space of a 2-category}

\noindent
Given $X$ a topological space with basepoint $p$, let $P^pX\subset X^I$ be the space of paths in $X$ that end at $p$. The map $\pi:P^pX\to X$ that sends a path $\gamma$ to its source $\gamma(0)$ is the well-known {\em path fibration}. Its fiber is the loop space $\Omega_p X$. 
$$\Omega_pX \to P^pX\xto \pi X$$
Since $P^pX$ is contractible, it follows from the sequence of homotopy groups induced by $\pi$ that $\Omega_p X$ is a homotopy-theoretic shift of $X$.

We shall construct a categorical analogue to the path fibration and prove our main theorem, which provides an algebraic description of $\Omega_pB_2C$. Finally we give some simple examples and discuss the necessity of our hypothesis.

\bigskip

Throughout this section $C$ is a connected small 2-category and $c$ an object of $C$. By connected we mean that any two objects are linked by a chain of arrows.

\begin{definition}\label{path}
We define the {\em path functor} $L:(\DC)\op\laxto\cat$ as the lax functor which is the composition of $\sup{}\op$ with the representable 2-functor induced by $c$.
$$\xymatrix{ & C\op \ar[dr]^{h^c}& \\ (\DC)\op
\ar@{-->}[ur]^{\sup{}\op} \ar@{-->}[rr]_L& & \cat}$$
\end{definition}

We concentrate on $E=(L\rtimes(\DC)\op)\op$, namely the opposite category to the Grothendieck construction over $L$ (cf. example \ref{construction}).
Note that an object of $E$ can be regarded as a chain $x_0\to\dots\to x_n\to c$, and more generally, we can identify $E$ with a subcategory of $\DC$. Let us denote this inclusion by $i:E\to\DC$. Besides, $E$ is isomorphic to the homotopy fiber of  $\sup:\DC\laxto C$ over the object $c$ and therefore it is contractible (cf. proof of theorem \ref{2suplaxo}).

% defintion of the categorical path fibration
\begin{definition}
We define the {\em categorical path fibration} as the following diagram in $\cat$:
$$C(c,c)\op\to E \xto p \DC$$
where $p:E\to\DC$ is the canonical projection of the Grothendieck construction and $C(c,c)\op$ is identified with the fiber over $c$ (chain of length 0).
\end{definition}

%remark on the fibers and homotopy fibers of $p$
Note that the fibers of $p$ are $p^{-1}(x)=L(x)\op=C(x_n,c)\op$. As in any prefibration, the inclusion $p^{-1}(x)\to p/x$ of the actual fiber into the homotopy fiber admits a left adjoint and hence is a weak equivalence.

\medskip

%relation with the topological framework
Next we shall relate the categorical path fibration with the topological one. The projection $p:E\to \DC$ and the inclusion $i:E\to\DC$ are linked by a natural transformation $H:p\then i$, which on an object $x_0\to\dots\to x_n\to c$ is given by the inclusion of ordinals $[n]\to[n+1]$ and the trivial 2-cells.
Regarding $H$ as a functor $E\times I\to \DC$ and composing with $\sup$ we get a lax functor
$$\sup\circ H: E\times I \laxto C \quad \text{ which yields }\quad 
BE\times I \to B_2C$$ 
Here we are writing $I$ for both the interval in $\cat$ and in $\top$, applying $B_g$ to the lax functor and using the natural homotopy equivalence $B_gC\simeq B_2C$.
The exponential law induces a map $BE \to B_2C^I$, whose image lies in the space of paths that end at $c$, for $\sup\circ i$ is the constant functor $c$.

\begin{definition}
We denote by $\phi: BE \to P^cB_2C$ the map defined above, and refer to it as the {\em transition map} relating both the categorical and topological path fibrations.
\end{definition}

The transition map $\phi: BE \to P^cB_2C$ fits into the following diagram
$$\xymatrix{
B(C(c,c))\ar[r] \ar[d] & BE \ar[d]_\phi \ar[r]^{Bp} & B(\DC) \ar[d]_{B\sup} \\
\Omega_cB_2C \ar[r] &  P^cB_2C \ar[r]_\pi & B_2C}$$
where we are identifying $B(C(c,c))\cong B(C(c,c)\op)$ via the canonical 
homeomorphism, and the left arrow is the restriction of $\phi$ to the fibers.

\bigskip

If $X$ is a topological space and $P_p^qX$ denotes the space of paths in $X$ that start at $p$ and end at $q$, then every path $\gamma$ from $p$ to $p'$ induces a homotopy equivalence $P_{p'}^qX\xto\simeq P_p^qX$, $\alpha\mapsto \gamma\ast\alpha$, where $\ast$ stands for the composition of paths.
On the other hand, if $C$ is a connected 2-category then the homotopy type of the hom-categories $C(c,c')$ might vary.
This motivates the following requirement.

\begin{definition}
We shall say that the pair $(C,c)$ satisfies the {\em condition $Q$} if for every arrow $f:c'\to c''$ the functor $f^*:C(c'',c)\to C(c',c)$ is a weak equivalence in $\cat$.
\end{definition}

Now we can state our main theorem.

\begin{theorem}\label{main}
Let $(C,c)$ be a 2-category satisfying condition $Q$. Then the category of endomorphisms $C(c,c)$ has the homotopy type of the loop space of $B_2C$ with base-point $c$, say
$$B(C(c,c))\simeq \Omega_c B_2C$$
\end{theorem}

\begin{proof}
%Every topological map $X\to Y$ induces a long exact sequence of homotopy groups relating those of the homotopy fiber, $X$ and $Y$. 
By condition $Q$, the path functor $L$ maps every arrow of $\DC$ to a weak equivalence. It follows that the base-change functors of the prefibration $E\to \DC$ are weak equivalences and therefore the hypothesis of Theorem B is fulfilled (cf. theorem \ref{thmb}). Thus, in the long exact sequence of homotopy groups arising from $Bp:BE\to B(\DC)$ we can identify those of the homotopy fiber with those of $B(C(c,c))$.

The transition map $\phi$ and the naturality allow us to compare the long exact sequences of homotopy groups coming from both the categorical and the topological path fibrations. Since $B(\sup)$ and $\phi$ are weak equivalences, it follows from the five lemma that 
$B(C(c,c))\xto\simeq \Omega_cB_2C$
 is a weak equivalence as well. Since these spaces have the homotopy type of a CW-complex (Milnor's classical theorem on spaces of maps), it is a homotopy equivalence.
\end{proof}

\begin{example}
If $C$ is a groupoid (all arrows invertible, only trivial 2-cells), then it clearly satisfies condition $Q$. The classifying space of $C$ is an Eilenberg-MacLane space $K(G,1)$, where $G$ is the group of automorphisms of a given object. In this case the loop space $\Omega_cBC$ has the homotopy type of the discrete set $G$, for each of its components is contractible.
\end{example}

\begin{example}
More generally, if $C$ is a 2-groupoid (all arrows and 2-cells invertible), then condition $Q$ is fulfilled. By theorem \ref{main} $\Omega_cB_2C$ is the classifying space of a groupoid, hence a $K(G,1)$. It follows that $B_2C$ is a homotopy 2-type.
\end{example}

\begin{example}
We present a  minimalistic example that shows that condition $Q$ is a sufficient but not a necessary condition. Let $C$ be the 1-category with three objects and two nontrivial arrows
%$$\clubsuit \from \heartsuit \to \spadesuit$$
$$C =\left\{c \from c' \xto f c''\right\}$$
It is clear that $B_2C\cong BC\simeq\ast$ and hence $\Omega_cB_2C\simeq \ast\cong C(c,c) $. Despite that, condition $Q$ does not hold, for the map $f^\ast:C(c'',c)=\emptyset \to C(c',c)= \ast$ is not a weak equivalence.
\end{example}

\bigskip

Given $c'$ an object of $C$, there is a canonical map $B(C(c',c))\to P_{c'}^cB_2C$ relating the algebraic and geometric paths. It is obtained by identifying $B(C(c',c))$ with the fiber of $Bp$ over $c'$ and then applying the transition map $\phi$. This map is natural in the following sense.

\begin{lemma}
If $f:c'\to c''$ is an arrow in $C$ then the  diagram of spaces
$$\xymatrix{
B(C(c'',c)) \ar[d]_{B(f^*)} \ar[r] & P_{c''}^cB_2C \ar[d]^{f\ast-}\\
B(C(c',c))  \ar[r] & P_{c'}^cB_2C}$$
commutes up to homotopy, where the horizontal arrows are the canonical maps and ${f\ast-}$ assigns to a path $\gamma$ the composition $f\ast\gamma$ with the path given by $f$.
\end{lemma}

\begin{proof}
For each point $p\in B(C(c'',c))$ we have two paths in $B_2C$ from $c'$ to $c$. We can deform these paths linearly one into the other. Actually, if $p$ belongs to the cell indexed by the $n$-simplex $g_0\then\cdots\then g_n$ of $N(C(c'',c))$, then the two corresponding paths lie in the cell indexed by the $(n,2)$-bisimplex
$(f\then \cdots\then f,g_0\then\cdots\then g_n)$, whose first coordinate is constant. This deformation is well-defined because the simplices are glued with linear maps, and is clearly continuous.
\end{proof}

\begin{corollary}
Under the hypothesis of theorem \ref{main}, the map $B(C(c',c))\to P_{c'}^cB_2C$ is a weak equivalence for all $c'$.
\end{corollary}

\begin{proof}
Take $c''=c$ in the previous lemma. Note that if $C$ is connected then for all $c'$ there must exist one arrow $f:c'\to c$ because of $Q$. In the square of the lemma the upper map is a weak equivalence by \ref{main}. The left one is so by hypothesis and the right one is always a weak equivalence. Then the bottom one is so by a two-out-of-three argument: if in a commutative triangle two maps are weak equivalences, then so does the third.
\end{proof}

Now we can state a partial converse for our main result. 

\begin{corollary}[Partial converse for theorem \ref{main}]
Let $C$ be a 2-category and $c$ an object of $C$. If for every $c'$ the canonical map $B(C(c',c))\to P_{c'}^cB_2C$ is a weak equivalence, then the pair $(C,c)$ satisfies condition $Q$.
\end{corollary}

\begin{proof}
It follows from the lemma and the two-out-of-three argument.
\end{proof}

%%%%%%%%%%%%%%%%%%%%%%%
%%%%%%%%%%%%%%%%%%%%%%%

\section{Delooping}

A 2-category $M$ with a single object is the same as a strict monoidal category. In this section we recover a classical result on delooping classifying spaces of monoidal categories from our theorem \ref{main}.

\bigskip

A (small strict) monoidal category is a monoid object in $\cat$. It consists of a small category $M$ together with an associative product  $\tensor:M\times M\to M$, $(x,y)\mapsto x\tensor y$, and a unit object $1$. Given $M$ a monoidal category, we denote by $\overline M$ its associated 2-category (cf. example \ref{overline}).
Note that the bar resolution of $M$ equals the nerve $\u N(\overline M)$.

Since the nerve functor preserves products, the classifying space $BM$ inherits a monoid structure in a natural way. Next we shall give a necessary and sufficient condition to ensure that $BM$ is a loop space.

\begin{proposition}\label{inverso}
Let $(M,\tensor)$ be a monoidal category. The following are equivalent:
\begin{enumerate}\itemsep=0pt
\item[a)] the topological monoid $BM$ admits an inverse up to homotopy;
\item[b)] the functors $r_x:M\to M$, $y\mapsto y\tensor x$ are weak equivalences;
\item[c)] the functors $l_x:M\to M$, $y\mapsto x\tensor y$ are weak equivalences;
\item[d)] $\pi_0(BM)$ is a group with the product induced by $\tensor$.
\end{enumerate}
If these hold, then the space $BM$ has the homotopy type of a loop space (is deloopable).
\end{proposition}

\begin{proof}
Clearly $a)\then b)$, $a)\then c)$ and $a)\then d)$. The proof of $d)\then a)$ can be found in \cite{segal2}. We shall prove that $b)\then a)$, which is analogous to $c)\then a)$. 

Consider $(\tensor,\pr_2):M\times M\to M\times M$ as a map over $M$.
$$\xymatrix@C=10pt{ M\times M \ar[rr]^{(\tensor,\pr_2)}  \ar[rd]_{\pr_2}& & M\times M \ar[ld]^{\pr_2}\\  & M}$$
If $x$ is an object in the base, then the map between the fibers can be identified with $r_x$, which is a weak equivalence by hypothesis. Since projections are prefibrations, the map between the homotopy fibers is also a weak equivalence and we conclude that $(\tensor,\pr_2)$ is a weak equivalence by a relative version of Theorem A (see for example \cite{dh2}). The rest is routine: if $(v_1,v_2):BM\times BM\to BM\times BM$ is an inverse for $B(\tensor,\pr_2)$, then $v_2\simeq\pr_2$ and the inverse up to homotopy for the monoid $BM$ is the composition
$$BM\xto{\cte_1\times\id} BM\times BM \xto{v_1} BM.$$

%\end{proof}

%\begin{proposition}\label{delooping}
%Let $(M,\tensor)$ be a monoidal category. The space $BM$ is deloopeable if and only if $(BM,\tensor)$ has an inverse up to homotopy.
%\end{proposition}

Let us now prove the last assertion. Note that if $BM$ is the space of loops of another space, then it has an inverse indeed. On the other hand, if $BM$ admits an inverse up to homotopy, then the functors $r_x:M\to M$, $y\mapsto y\otimes x$ are weak equivalences, the 2-category $\overline M$ satisfies condition $Q$ and hence we can apply theorem \ref{main} which gives
$$BM\simeq \Omega B_2\overline M.$$
\end{proof}

\end{document}